\documentclass[11pt, reqno]{amsart} 
\usepackage{subfiles} 

\usepackage{amsfonts,amsmath,amssymb,amscd}
\usepackage{amsthm}
\usepackage{latexsym}
\usepackage[dvipdfmx]{graphicx}
\usepackage{url}
\usepackage[normalem]{ulem}
\usepackage{psfrag}
\usepackage{color}
\usepackage{xcolor}
\usepackage{comment}
\usepackage[colorinlistoftodos]{todonotes}

\usepackage{thmtools} 
\usepackage{thm-restate} 

\theoremstyle{plain}
\everydisplay{\displaystyle}
\newtheorem{theorem}{Theorem}[section]
\newtheorem{lemma}[theorem]{Lemma}

\newtheorem{proposition}[theorem]{Proposition}

\theoremstyle{definition}
\newtheorem{definition}[theorem]{Definition}

\newtheorem{remark}[theorem]{Remark}

\theoremstyle{definition}

\newcommand{\Z}{\mathbb{Z}}

\newcommand{\Q}{\mathbb{Q}}

\newcommand{\GL}{GL}

\newcommand{\colvec}[1]{\begin{bmatrix}#1\end{bmatrix}}
\definecolor{myred}{rgb}{.8,.0,.0}
\definecolor{mygreen}{rgb}{.0,.6,.0}
\definecolor{mygray}{gray}{0.7}

\newif\ifshowred
\showredtrue 

\newif\ifreview
\reviewtrue 

\newcounter{reviewnote}

\newcounter{mystepcount}

\title[Stein fillings with relative trisection genus 2]{Stein fillings of planar contact 3-manifolds with relative trisection genus 2}
\author{Nobutaka Asano}
\address{National Institute of Technology, Tsuyama College, 624-1 Numa, Tsuyama-shi, Okayama, 708-8509, Japan}
\email{asano-n@tsuyama-ct.ac.jp}
\author{Natsuya Takahashi}
\address{Department of Mathematics, College of Science and Technology, Nihon University, 1-8-14, Kanda-Surugadai, Chiyoda-ku, Tokyo 101-8308, Japan}
\email{takahashi.natsuya@nihon-u.ac.jp}
\date{}
\begin{document}
\begin{abstract}
We study Stein fillings of planar contact $3$-manifolds under certain constraints on their relative trisections.
We partially classify the diffeomorphism types of such fillings admitting relative trisections with genus at most $2$.
\end{abstract}
\maketitle

\section{Introduction}

Gay and Kirby~\cite{GK16} introduced a trisection, which is a decomposition of a closed $4$-manifold into three $4$-dimensional $1$-handlebodies. 
The triple intersection of the three pieces is a closed surface, called the central surface of the trisection.
A trisection is a $4$-dimensional analogue of a Heegaard splitting of a $3$-manifold.
For a compact $4$-manifold with connected boundary, Gay and Kirby~\cite{GK16} also introduced a relative trisection, whose triple intersection is a compact surface with boundary.
A relative trisection canonically induces an open book decomposition on the boundary $3$-manifold.

The trisection genus is a nonnegative integer-valued invariant of a smooth $4$-manifold.
For a $4$-manifold $X$, the trisection genus $g(X)$ is defined as the minimal genus of the central surface among all (relative) trisections of $X$.
This invariant is a $4$-dimensional analogue of the Heegaard genus.
The classification of $4$-manifolds by trisection genus is a natural problem in trisection theory.
Meier and Zupan~\cite{MZ17} showed that every closed $4$-manifold of trisection genus at most $2$ is diffeomorphic to one of finitely many standard $4$-manifolds.

In this paper, we study the classification of $4$-manifolds with boundary by their relative trisection genus.
It is straightforward to see that $4$-manifolds with boundary whose relative trisection genus is $0$ are diffeomorphic to $4$-dimensional $1$-handlebodies.
However, the classification in relative trisection genus at least $1$ remains open.

For $4$-manifolds with boundary, the classification problem is also generally difficult.
We therefore restrict our attention to Stein surfaces, which form an important class of smooth $4$-manifolds with boundary in symplectic and contact topology. 
By Giroux correspondence, contact structures on closed $3$-manifolds correspond
to open book decompositions up to positive stabilization~\cite{G02}.
In this paper, we focus on Stein surfaces that induce contact $3$-manifolds admitting a planar open book decomposition (i.e., one with genus-$0$ pages). 
It was shown by Wendl~\cite{W10} that Stein fillings of planar contact
$3$-manifolds admit genus-$0$ positive allowable Lefschetz fibrations compatible with the supporting open book decomposition on the boundary.
Furthermore, we assume that the pages of these open book decompositions arise from relative trisections.
Under these assumptions, %
we show that any such Stein surface with relative trisection genus at most $1$ is diffeomorphic to a $4$-dimensional $1$-handlebody.

\begin{restatable}{theorem}{maingzeroone}\label{thm:g=01}
Let $X$ be a Stein filling of some contact $3$-manifold supported by a planar open book decomposition with $b$ binding components.
\begin{itemize}
\item $X$ admits a $(0,k;0,b)$-relative trisection if and only if $k=b-1$ and $X$ is diffeomorphic to $\natural^{b-1}(S^1 \times D^3)$.
\item $X$ admits a $(1,k;0,b)$-relative trisection if and only if $b\geq2$, $k=b-1$, and $X$ is diffeomorphic to $\natural^{b-2}(S^1 \times D^3)$.
\end{itemize}
\end{restatable}
Here a $(g,k;p,b)$-relative trisection means that the central surface has genus $g$, the sectors have genus $k$, and the induced boundary open book page has genus $p$ with $b$ boundary components.

The genus-$2$ case is more complicated. The following theorem gives a classification when the number of binding components of the open book decomposition is at most $4$.

\begin{restatable}{theorem}{maingtwo}\label{thm:g=2}
Let $X$ be a Stein filling of some contact $3$-manifold supported by a planar open book decomposition with $b$ binding components.
\begin{itemize}
\item If $X$ admits a $(2,k;0,b)$-relative trisection, then $b\geq 2$ and $k=b-1$.
\item $X$ admits a $(2,1;0,2)$-relative trisection if and only if $X$ is diffeomorphic to $E_{-2}$, which is the $D^2$-bundle over $S^2$ with Euler number~$-2$.
\item $X$ admits a $(2,2;0,3)$-relative trisection if and only if $X$ is diffeomorphic to either $E_{-2}\natural(S^1\times D^3)$ or $D^4$.
\item $X$ admits a $(2,3;0,4)$-relative trisection if and only if $X$ is diffeomorphic to $V_{p,q}$ for some coprime integers $p$ and $q$. Here, $V_{p,q}$ is the total space of the positive allowable Lefschetz fibration depicted in Figure~\ref{fig:V_pq_pillowcase_Sec1}.
\end{itemize}
\end{restatable}
\begin{figure}[!htbp]
\centering
\includegraphics[scale=0.4]{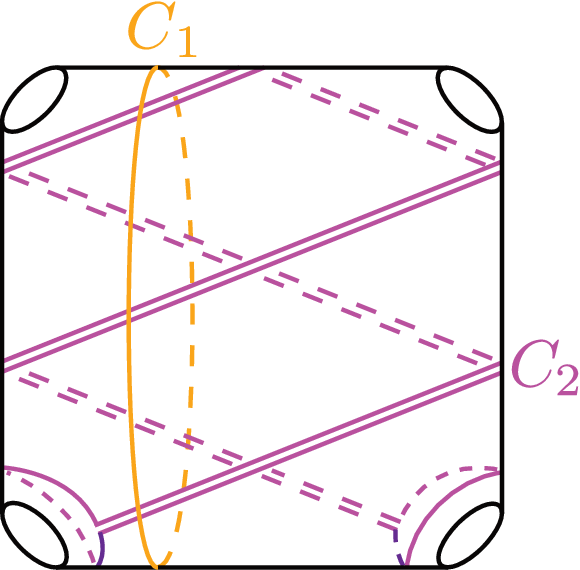}
\caption{A diagram of the positive allowable Lefschetz fibration $V_{p,q}$. $C_1$ and $C_2$ are vanishing cycles on the regular fiber $\Sigma_{0,4}$. The slope of $C_2$ is $p/q=2/5$ in this figure.}
\label{fig:V_pq_pillowcase_Sec1}
\end{figure}

The last case of Theorem~\ref{thm:g=2} gives rise to a family of Stein surfaces $V_{p,q}$ indexed by slopes $p/q$ on the four-punctured sphere. We normalize $q>0$ and $\gcd(p,q)=1$.
For details of $V_{p,q}$, see Section~\ref{sect:proofmain}.

We outline the proofs of Theorems~\ref{thm:g=01} and \ref{thm:g=2}.
By the result of Wendl~\cite{W10}, any Stein filling of a planar contact $3$-manifold admits a genus-$0$ positive allowable Lefschetz fibration that induces the supporting open book decomposition.
Using this theorem, we obtain a constraint on the number of vanishing cycles on the fiber surface.
Since the topological type of the regular fiber is fixed by our assumptions, the problem reduces to classifying ordered configurations of simple closed curves on the fiber surface.
We also show that the Stein surfaces $V_{p,q}$ cannot admit a relative trisection of genus less than $2$ if $q \neq 1$ (see Proposition~\ref{prop:reltrisection_genus}). 

The following proposition gives a sufficient condition for  $V_{p,q}$ and $V_{p',q'}$ to be diffeomorphic.

\begin{restatable}{proposition}{mainVpq}\label{prop:homeo_on_fiber}
Let $p,p',q, q'$ be integers satisfying $q, q'>0$ and $\gcd(p,q)=\gcd(p',q')=1$.
If $q=q'$ and either $p\equiv \pm p' \pmod{q}$ or $pp'\equiv \pm 1 \pmod{q}$, then $V_{p,q}$ is diffeomorphic to $V_{p',q'}$.
\end{restatable}

In particular, when $q$ is odd, the converse of Proposition~\ref{prop:homeo_on_fiber} holds. 

\begin{restatable}{theorem}{mainclass}\label{thm:classification}
Suppose that $q$ and $q'$ are odd. Then,
$V_{p,q}$ is diffeomorphic to $V_{p',q'}$
if and only if $q=q'$ and either $p\equiv \pm p' \pmod{q}$ or $pp'\equiv \pm 1 \pmod{q}$.
\end{restatable}

\subsection*{Conventions}\label{sec:conventions}

We fix the following conventions and notation throughout the paper.
All manifolds are assumed to be compact, oriented, and smooth.
All diffeomorphisms are orientation-preserving unless otherwise stated.
All diffeomorphisms of surfaces are assumed to fix the boundary pointwise.
If there exists an orientation-preserving diffeomorphism between two manifolds $X$ and $Y$, then we write $X \cong Y$.
We also use the symbol $\cong$ to denote isomorphisms of groups.
For integers $g\geq0$ and $b\geq1$, $\Sigma_{g,b}$ denotes a compact, connected, oriented surface of genus $g$ with $b$ boundary components.
For a simple closed curve $c$ on a surface, $t_c$ denotes the right-handed Dehn twist along $c$.

\section{Preliminaries}
\subsection{Contact 3-manifolds and open book decompositions}

\begin{definition}
Let $L$ be a link in a closed $3$-manifold $M$, and let $\pi : M \setminus L \to S^1$ be a locally trivial fiber bundle.
The pair $(L, \pi)$ is called an \textit{open book decomposition} of $M$ if for each $\theta \in S^1$, the fiber $\pi^{-1}(\theta)$ is the interior of a compact surface with boundary $\partial \overline{\pi^{-1}(\theta)} = L$.
Moreover, on a tubular neighborhood $\nu(L)\cong L\times D^2$, the map $\pi$ agrees with the angular coordinate on $D^2$.
The link $L$ is called the \textit{binding}, and the fibers $\pi^{-1}(\theta)$ are called the \textit{pages}.
\end{definition}

Let $\Sigma$ be a compact oriented surface with boundary and $\phi$ a self-diffeomorphism of $\Sigma$.
One can reconstruct a closed $3$-manifold together with an open book decomposition from the pair $(\Sigma, \phi)$.

\begin{definition}
Let $\Sigma$ be a compact oriented surface with boundary.
%
%
Let $\phi:\Sigma\to\Sigma$ be an orientation-preserving diffeomorphism, called the \textit{monodromy}, such that $\phi$ is the identity in a neighborhood of $\partial{\Sigma}$.
Then the pair $(\Sigma,\phi)$ is called an \textit{abstract open book decomposition}.
\end{definition}

The next remark explains how $(\Sigma,\phi)$ determines a closed $3$-manifold $M_\phi$ together with an open book decomposition with page $\Sigma$. 

\begin{remark}
From an abstract open book decomposition, we can construct a closed $3$-manifold and its open book decomposition.
Let $(\Sigma,\phi)$ be an abstract open book decomposition whose boundary has $n$ components: $\partial{\Sigma}=l_1 \sqcup l_2 \sqcup \cdots \sqcup l_n$.
Then we obtain the mapping torus $\Sigma_\phi := \Sigma \times [0, 1]/ \sim$, where the equivalence relation $\sim$ is given by $(\phi(x), 0) \sim (x, 1)$ for all $x \in \Sigma$.
The closed $3$-manifold $M_\phi$ is then defined as follows:
\begin{equation*}
M_\phi := {\Sigma_\phi}\ {\cup_\psi} \left( \coprod_{n} S^1 \times D^2 \right).
\end{equation*}
The map $\psi$ that attaches the $n$ solid tori $S^1\times D^2$ to $\Sigma_\phi$ is defined as the union of maps $\psi_1, \psi_2, \ldots, \psi_n$.
For each boundary component $l_i$ of $\partial{\Sigma}$, the attaching map $\psi_{i}:\partial(S^1 \times D^2)\to{l_i\times[0, 1]/\sim} \subset \Sigma_\phi$ is a diffeomorphism (unique up to isotopy) satisfying the meridian $\{p\} \times \partial D^2$ is mapped to $\{p'\}\times[0, 1]/\sim$, where $\{p \} \times \partial D^2 \subset S^1 \times \partial D^2$ and $\{p'\} \times ([0,1]/\sim)\ \subset l_i \times ([0,1]/\sim)$.
Let $L$ be the union of the cores $S^1\times \{0\}$ of the $n$ solid tori $S^1\times D^2$ attached to $\Sigma_\phi$
and $\pi : \Sigma_{\phi}\to S^1$ be the following map
\begin{equation*}
\pi : \Sigma_\phi = {\Sigma\times[0, 1]/\sim}\ \to \ [0, 1]/\sim\ \cong\  S^1, \ \text{where } \pi([(x,t)])=t.
\end{equation*}
Let $\tilde{\pi} : M_{\phi}\setminus L\to S^1$ be the natural extension of the map $\pi$.
Then the pair $(L,\tilde{\pi})$ is an open book decomposition of $M_{\phi}$.
\end{remark}
In this paper, we identify an open book decomposition $(L, \tilde{\pi})$ with the corresponding abstract open book decomposition $(\Sigma, \phi)$.
Open book decompositions are deeply related to $3$-dimensional contact geometry.
We next recall the definition of a contact structure. 

\begin{definition}
Let $M$ be a closed, oriented $3$-manifold.
A hyperplane field $\xi$ on $M$ is called a \textit{contact structure} if there exists a $1$-form $\alpha$ on $M$ such that $\xi = \ker(\alpha)$ and $\alpha_p \wedge (d_p\alpha) > 0$ for every $p\in M$.
The pair $(M, \xi)$ is called a \textit{contact $3$-manifold}.
\end{definition}

A contact structure $\xi$ on a closed $3$-manifold $M$ is said to be \textit{supported} by an open book decomposition $(L,\pi)$ of $M$ if $\xi = \mathrm{ker}\alpha$ for some contact form $\alpha \in \Omega^1(M)$ such that $\alpha$ is positive on the binding and $d\alpha$ restricts to a positive area form on each page. 
The following one-to-one correspondence is well known.
\begin{equation*}
  \frac{\{ \text{contact structures on}\ M \}}{\text{isotopy}}
  \overset{1:1}\longleftrightarrow
  \frac{\{ \text{open book decompositions of}\ M \}}{\text{positive stabilization}}.
\end{equation*}
This is called the Giroux correspondence \cite{G02}.
In this paper, a particularly important class is that of planar open books, since many fillability and classification results become sharper in the planar setting.

\begin{definition}
An open book $(\Sigma,\phi)$ is called \textit{planar} if the page $\Sigma$ is diffeomorphic to a punctured sphere $\Sigma_{0,b}$.
A contact $3$-manifold $(M,\xi)$ is called \textit{planar} if $(M,\xi)$ is supported by a planar open book decomposition.
\end{definition}
We say that a contact structure $\xi$ on a $3$-manifold $M$ is \textit{Stein fillable} if there exists a Stein surface $X$ with $\partial X=M$ such that the induced contact structure on $\partial X$ is contactomorphic to $\xi$. 
Such a Stein surface $X$ is called a \textit{Stein filling} of $(M,\xi)$.
\subsection{Lefschetz fibrations}
\label{subsec:Lefschetz fibrations}

We now recall the basic definitions of Lefschetz fibrations and the associated handle decompositions.
\begin{definition}
Let $X$ be a compact, oriented, smooth $4$-manifold with boundary.
A smooth map $f:X \to D^2$ is called a \textit{positive Lefschetz fibration} over $D^2$ if it satisfies the following conditions:
\begin{itemize}
\item
There exist finitely many points $b_1, b_2, \ldots, b_n$ in the interior of $D^2$ such that the restriction map
\begin{equation*}
f : {f^{-1}(D^2 \setminus \{ b_1, b_2, \ldots, b_n \})} \to {D^2\setminus\{ b_1, b_2, \ldots, b_n \}}
\end{equation*}
 is a fiber bundle over $D^2\setminus\{ b_1, b_2, \ldots, b_n \}$ with fiber diffeomorphic to an oriented surface.
\item
The points $b_1, b_2, \ldots, b_n$ are the critical values of $f$, and for each $i\in\{1,2,\ldots,n\}$, the fiber $f^{-1}(b_i)$ contains a unique critical point $p_i$.
\item
For each critical value $b_i$ and corresponding critical point $p_i$, there are local complex coordinate charts, compatible with the orientations of $X$ and $D^2$ such that $f$ is locally written as $f (z_1, z_2) = z_1^2 + z_2^2$.
%
%
\end{itemize}
A fiber $f^{-1}(b)$ is called \textit{singular} if $b\in\{b_1, b_2, \ldots, b_n\}$, and is called \textit{regular} otherwise.
\end{definition}

A positive Lefschetz fibration over $D^2$ yields a handle decomposition of the total space.
Let $f : X \to D^2$ be a positive Lefschetz fibration whose regular fibers are diffeomorphic to the surface $\Sigma_{g,b}$. 
Then $X$ admits a handle decomposition:
\[
\begin{aligned}
X &= (D^2 \times \Sigma_{g,b}) \cup \left( \bigcup_{i=1}^{n} h^{(2)}_i \right) \\[3mm]
  &= h^{(0)} \cup \left( \bigcup_{j=1}^{2g+b-1} h^{(1)}_j \right)
    \cup \left( \bigcup_{i=1}^{n} h^{(2)}_i \right),
\end{aligned}
\]
where each $h^{(k)}_l$ denotes the $l$-th $k$-handle.
Each $2$-handle $h^{(2)}_i$ corresponds to the critical point $p_i$ and is attached along a simple closed curve $\alpha_i \subset \{\mathrm{pt}.\} \times \Sigma_{g,b} \subset D^2 \times \Sigma_{g,b}$.

The $2$-handle $h_i^{(2)}$ is attached along $\alpha_i$ with framing $-1$ relative to the surface framing.
The attaching circle $\alpha_i$ of $h^{(2)}_i$ is called a \textit{vanishing cycle}. 
The ordered collection $(t_{\alpha_1}, t_{\alpha_2}, \ldots, t_{\alpha_n})$ of right-handed Dehn twists is called a \textit{monodromy factorization} of the positive Lefschetz fibration $f$.
Conversely, for any positive Lefschetz fibration $f: X \to D^2$, there exists an ordered set of simple closed curves $(\alpha_1, \alpha_2, \ldots, \alpha_n)$ on a surface $\Sigma$ such that $f$ is obtained from this ordered set.
Let $X(\Sigma; \alpha_1, \alpha_2, \ldots, \alpha_n)$ denote the total space of the Lefschetz fibration determined by $(\Sigma; \alpha_1, \alpha_2, \ldots, \alpha_n)$.

\begin{definition}
A positive Lefschetz fibration $f: X \to D^2$ is said to be \textit{allowable} if all of its vanishing cycles $\alpha_i$ are homologically nontrivial in the fiber $\Sigma_{g,b}$.
For brevity, we call a positive allowable Lefschetz fibration a \textit{PALF}.
\end{definition}

Loi and Piergallini~\cite{LP01} and Akbulut and Ozbagci~\cite{AO01} showed that every Stein surface admits a PALF $f : X \to D^2$ with regular fiber $\Sigma$.
In particular, the contact structure induced on $\partial X$ is supported by the open book decomposition induced by $f$. 
The page of this open book is diffeomorphic to $\Sigma$, and its monodromy admits a positive factorization into Dehn twists along the vanishing cycles of $f$. 
We use the following theorem to prove the main results of this paper.

\begin{theorem}[Wendl~\cite{W10}]\label{thm-Wendl}
Let $(M, \xi)$ be a contact $3$-manifold supported by a planar open book $(\Sigma_{0,b}, \phi)$.
Let $X$ be a Stein filling of $(M,\xi)$.
Then, $X$ is diffeomorphic to the total space of a PALF obtained from $(\Sigma_{0,b}, \phi)$.
\end{theorem}

\begin{remark}
 By a PALF obtained from $(\Sigma_{0,b}, \phi)$, we mean a positive allowable Lefschetz fibration whose regular fiber is $\Sigma_{0,b}$ and whose vanishing cycles realize a factorization of the monodromy $\phi$ into right-handed Dehn twists.
 Such a factorization is not unique in general, and hence the resulting PALF is not necessarily unique either.
\end{remark}


\subsection{Relative trisections}

A trisection is a decomposition of a closed $4$-manifold $X$ into three pieces: $X = X_1 \cup X_2 \cup X_3$, where each $X_i$ is a $4$-dimensional $1$-handlebody.
The pairwise intersections $X_i \cap X_j$ are $3$-dimensional $1$-handlebodies, and the triple intersection $X_1 \cap X_2 \cap X_3$ is a surface.

We are particularly interested in the corresponding notion for $4$-manifolds with boundary, known as relative trisections.
Let $X$ be a smooth, compact, connected, and oriented $4$-manifold with connected boundary, and let $g$, $k$, $p$ be non-negative integers and $b$ a positive integer.
Suppose that $X$ admits a $(g, k; p, b)$-{relative trisection} $X = X_1 \cup X_2 \cup X_3$. Then the following conditions hold:
\begin{itemize}
\item The parameters $g, k, p$ and $b$ satisfy the inequality $2p + b - 1 \leq k \leq g + p + b - 1$.
\item Each $X_i$ is diffeomorphic to a $4$-dimensional $1$-handlebody of genus $k$.
\item Each double intersection $X_i \cap X_j$ (equivalently, $\partial X_i \cap \partial X_j$) is a compression body from $\Sigma_{g,b}$ to $\Sigma_{p,b}$. 
\item The triple intersection $X_1 \cap X_2 \cap X_3$ is diffeomorphic to the surface $\Sigma_{g,b}$.
\item The boundary $\partial X$ admits an open book decomposition with page $\Sigma_{p,b}$.
\end{itemize}

For detailed definitions of relative trisections, we refer the reader to \cite{CGPC18}.

\begin{theorem}[Gay--Kirby~\cite{GK16}]
For any $4$-manifold $X$ with connected boundary and any open book decomposition of $\partial{X}$, there exists a relative trisection of $X$ that induces the given open book decomposition.
\end{theorem}

The genus of a relative trisection $X=X_1\cup X_2\cup X_3$ is the genus of the triple intersection surface $X_1\cap X_2\cap X_3$.
The trisection genus $g(X)$ of a $4$-manifold $X$ is defined as follows:
\begin{equation*}
g(X) := \min\{g\in\Z_{\geq0} \mid \text{$X$ admits a $(g,k;p,b)$-relative trisection}\}.
\end{equation*}

The trisection genus is a diffeomorphism invariant of smooth $4$-manifolds.
Closed $4$-manifolds with trisection genus $2$ have been classified~\cite{MZ17}.

\begin{proposition}[Castro--Ozbagci~{\cite{CO19}}]\label{prop:Euler-rel}
If a $4$-manifold $X$ with boundary admits a $(g,k;p,b)$-relative trisection, then $\chi(X)=g-3k+3p+2b-1$.
\end{proposition}

A relative trisection can be described combinatorially by three families of curves on a surface.

\begin{definition}
Let $X=X_1\cup X_2\cup X_3$ be a $(g,k;p,b)$-relative trisection of a $4$-manifold $X$ with boundary.
Set $\Sigma:=X_1\cap X_2\cap X_3$, and let $\alpha$, $\beta$, and $\gamma$ be collections of simple closed curves on $\Sigma$ that bound complete disk systems of the compression bodies $X_3\cap X_1$, $X_1\cap X_2$, and $X_2\cap X_3$, respectively.
Then we call the $4$-tuple $(\Sigma;\alpha,\beta,\gamma)$ a $(g,k;p,b)$-\textit{relative trisection diagram} of $X=X_1\cup X_2\cup X_3$.
\end{definition}

For a relative trisection diagram $(\Sigma;\alpha,\beta,\gamma)$, we draw the $\alpha$-, $\beta$-, and $\gamma$-curves in red, blue, and green, respectively.

\begin{figure}[!htbp]
\centering
\includegraphics[scale=0.75]{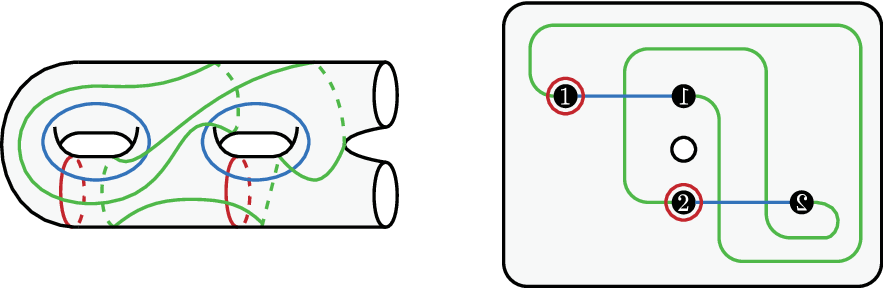}
\caption{An example of a relative trisection diagram. The diagram on the right is a planar representation of the left one.}
\label{fig:td-exam}
\end{figure}

It is known that a relative trisection can be constructed from a Lefschetz fibration on a $4$-manifold with boundary.

\begin{theorem}[Castro--Gay--Pinz\'{o}n-Caicedo~{\cite{CGPC18}}]
Let $f: X\to D^2$ be an achiral Lefschetz fibration with regular fiber $\Sigma_{p,b}$ and $n$ vanishing cycles. Then $X$ admits a $(p+n, 2p+b-1;p,b)$-relative trisection.
\end{theorem}

Moreover, a diagram describing vanishing cycles on a fiber surface can be converted into a relative trisection diagram.
For details, see Figure~14 in~\cite{CGPC18}.

\section{The proof of main theorem}\label{sect:proofmain}
The next proposition gives a constraint on the number of singular points of a PALF.

\begin{proposition}\label{prop:LF-gkpb}
Let $X$ be a Stein filling of a contact $3$-manifold supported by a planar open book decomposition with $b$ binding components.
Suppose that $X$ admits a $(g,k;0,b)$-relative trisection. 
Then the following hold:
\begin{itemize}
\item $b-1\leq k \leq b-1+\lfloor{g/3}\rfloor$.
\item 
There exists a PALF $f:X\to D^2$ such that the regular fiber is $\Sigma_{0,b}$ and the number of singular points is $g-3k+3b-3$.
\end{itemize}
\end{proposition}

\begin{proof}
By Theorem~\ref{thm-Wendl}, there exists a PALF $f : X \to D^2$ obtained from the open book decomposition with page $\Sigma_{0,b}$.  
Let $n$ denote the number of singular points of $f$.
Then $\chi(X) = n-b+2$ holds since $X$ admits a handle decomposition consisting of one $0$-handle, $b-1$ 1-handles, and $n$ 2-handles.
On the other hand, by Proposition~\ref{prop:Euler-rel}, we also have
$\chi(X) = g-3k+2b-1$.
Comparing the two expressions, we obtain
$n = g-3k+3b-3$.
Since the number of singular points must be non-negative, it follows that $g-3k+3b-3\geq0$, which implies $k\leq b-1 + g/3$.
From the definition of a relative trisection, we also have the inequality $2p+b-1\leq k\leq g+p+b-1$.
In the case $p=0$, this reduces to $b-1\leq k\leq g+b-1$.
Combining these two bounds, we obtain that $b-1\leq k\leq b-1+\lfloor{g/3}\rfloor$.
\end{proof}

We classify the diffeomorphism types of Stein fillings for $g=0,1$.

\maingzeroone*

\begin{proof} 
By the $g=0$ case of Proposition~\ref{prop:LF-gkpb}, the 4-manifold $X$ admits a $(0,b-1;0,b)$-relative trisection, and there exists a PALF $f:X\to D^2$ such that the regular fiber is $\Sigma_{0,b}$ and the number of singular points is $0$. Since $f$ has no vanishing cycles, the total space is diffeomorphic to $\natural^{b-1}(S^1\times D^3$). 
Conversely, $\natural^{b-1}(S^1\times D^3)$ admits a $(0,b-1;0,b)$-relative trisection.

By the $g=1$ case of Proposition~\ref{prop:LF-gkpb}, the 4-manifold $X$ admits a $(1,b-1;0,b)$-relative trisection, and there exists a PALF $f:X\to D^2$ such that the regular fiber is $\Sigma_{0,b}$ and the number of singular points is $1$. Let $C_1\subset\Sigma_{0,b}$ be the vanishing cycle of $f$.

If $b=1$, then $C_1$ bounds a disk in $\Sigma_{0,b}$, so $f$ cannot be a PALF.
Therefore, we must have $b\geq2$.
The vanishing cycle $C_1$ is as shown in Figure~\ref{fig:td-g=1}.
Thus, the total space of $f$ is diffeomorphic to $\natural^{b-2}(S^1\times D^3)$ by Kirby calculus.

Conversely, $\natural^{b-2}(S^1\times D^3)$ admits a genus-$1$ relative trisection.
Castro, Gay, and Pinz\'{o}n-Caicedo~\cite{CGPC18} gave an algorithm for constructing a relative trisection diagram from an achiral Lefschetz fibration.
Applying this algorithm, we obtain the $(1,b-1;0,b)$-relative trisection diagram from the PALF $(\Sigma_{0,b-1}; C_1)$.
The resulting diagram is shown in Figure~\ref{fig:td-g=1}(right).
\end{proof}

\begin{figure}[!htbp]
\centering
\includegraphics[scale=0.75]{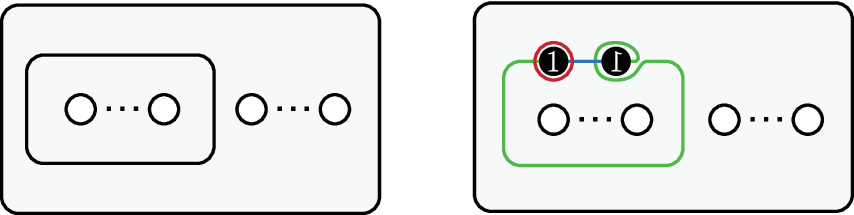}
\caption{Left: The vanishing cycle $C_1$ on $\Sigma_{0,b}$, Right: The induced $(1,b-1;0,b)$-relative trisection diagram.}
\label{fig:td-g=1}
\end{figure}

The following theorem classifies Stein fillings for $g=2$ and $b\le 4$.

\maingtwo*

We now explain the definition of $V_{p,q}$.
Fix the convention that the vertical and horizontal curves on $\Sigma_{0,4}$ have slopes $1/0$ and $0/1$, respectively (see Figure~\ref{fig:V_pq_pillowcase}~(left)). 
It is well known that isotopy classes of non-boundary-parallel simple closed curves on $\Sigma_{0,4}$ are parameterized by slopes in $\Q \cup \{1/0\}$ (see e.g.~\cite{CM09,MZ17}).
We call this characterization of simple closed curves on $\Sigma_{0,4}$ the pillowcase description.
For coprime integers $p$ and $q>0$, if $C_1$ and $C_2$ have slopes $1/0$ and $p/q$, respectively, then we define $V_{p,q}$ to be the total space of the PALF $(\Sigma_{0,4}; C_1, C_2)$, that is,
\[ V_{p,q}:=X(\Sigma_{0,4};C_1,C_2). \]
Figure~\ref{fig:V_pq_pillowcase}~(right) represents the PALF $(\Sigma_{0,4};C_1,C_2)$ corresponding to $V_{2,5}$. 

\begin{figure}[!htbp]
\centering
\includegraphics[scale=0.40]{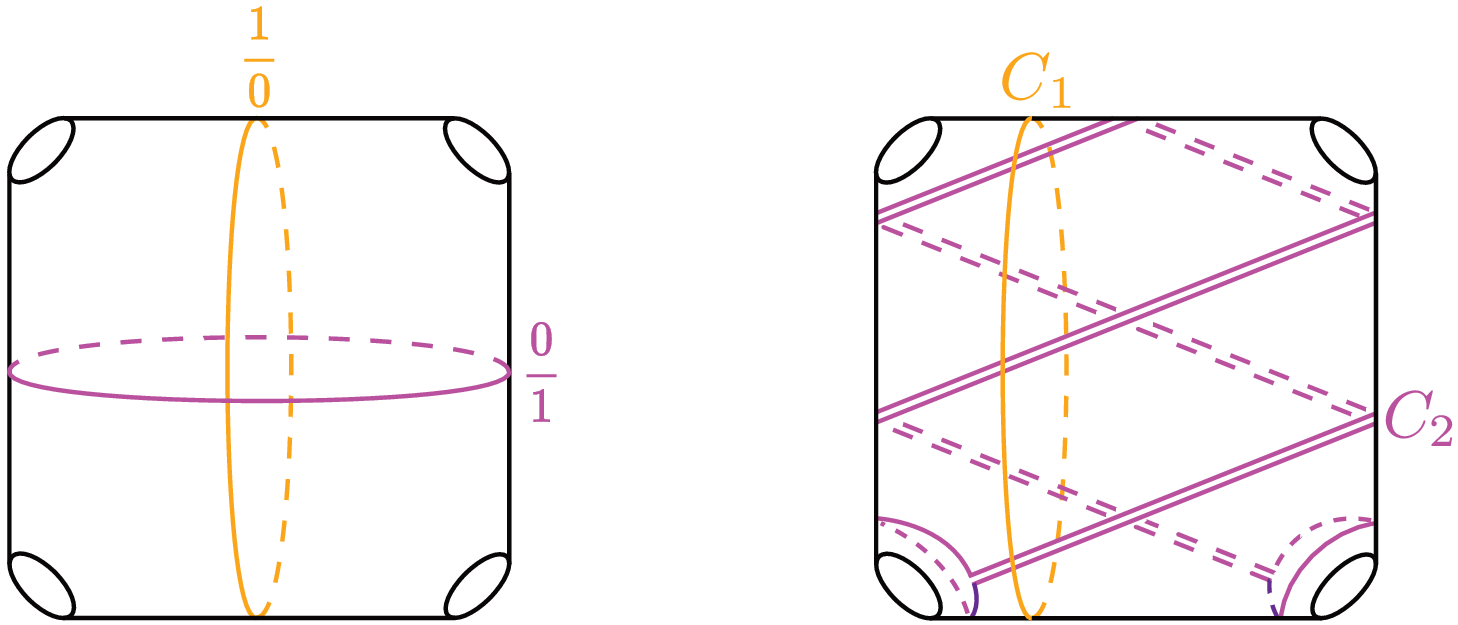}
\caption{Left: The vertical and horizontal curves on $\Sigma_{0,4}$. Right: A PALF $(\Sigma_{0,4}; C_1, C_2)$ representing $V_{2,5}$.}
\label{fig:V_pq_pillowcase}
\end{figure}

\begin{proof}[Proof of Theorem~\ref{thm:g=2}]
By the $g=2$ case of Proposition~\ref{prop:LF-gkpb}, it follows that $k=b-1$, and there exists a PALF $f:X\to D^2$ such that the regular fiber is $\Sigma_{0,b}$ and the number of singular points is $2$. Let $C_1, C_2 \subset \Sigma_{0,b}$ be the vanishing cycles of $f$.
If $b=1$, then both $C_1$ and $C_2$ are null-homologous in $\Sigma_{0,1}$, and thus $f$ cannot be a PALF. 
We must have $b\geq2$.

We consider the case $b=2$.
Since neither $C_1$ nor $C_2$ is null-homologous in $\Sigma_{0,2}$, $(\Sigma_{0,2};C_1,C_2)$ is as shown in Figure~\ref{fig:g=2b=2}~(left). 
Hence, the total space of the PALF $f$ has the Kirby diagram shown in Figure~\ref{fig:g=2b=2}~(middle), which represents $E_{-2}$.
Conversely, $E_{-2}$ admits a genus-$2$ relative trisection.
We can get a $(2,1; 0,2)$-relative trisection diagram from the PALF $(\Sigma_{0,2}; C_1, C_2)$. 
See Figure~\ref{fig:g=2b=2}~(right).
\begin{figure}[!htbp]
\centering
\includegraphics[scale=0.75]{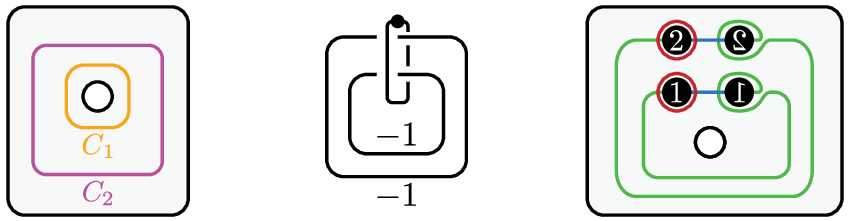}
\caption{Left: Vanishing cycles $C_1$ and $C_2$ on $\Sigma_{0,2}$, Middle: The induced handlebody diagram, Right: The induced $(2,1;0,2)$-relative trisection diagram.}
\label{fig:g=2b=2}
\end{figure}

Next, we consider the case $b=3$.
Since neither $C_1$ nor $C_2$ is null-homologous in $\Sigma_{0,3}$, the PALF $(\Sigma_{0,3};C_1,C_2)$ corresponds to one of the diagrams shown in  Figure~\ref{fig:g=2b=3a} (left) or Figure~\ref{fig:g=2b=3b}  (left).
For each case, the total space has the Kirby diagram shown in the middle, and the induced $(2,2;0,3)$-relative trisection diagram is depicted on the right.
The $4$-manifolds represented by Figure~\ref{fig:g=2b=3a} and Figure~\ref{fig:g=2b=3b} are $D^4$ and $E_{-2}\natural(S^1\times D^3)$, respectively.
\begin{figure}[!htbp]
\centering
\includegraphics[scale=0.75]{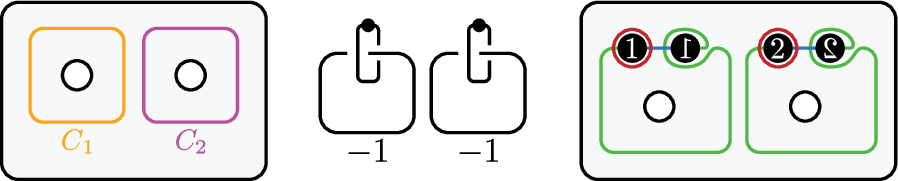}
\caption{Left: Vanishing cycles $C_1$ and $C_2$ on $\Sigma_{0,3}$, Middle: The induced handlebody diagram, Right: The induced $(2,2;0,3)$-relative trisection diagram.}
\label{fig:g=2b=3a}
\end{figure}
\begin{figure}[!htbp]
\centering
\includegraphics[scale=0.75]{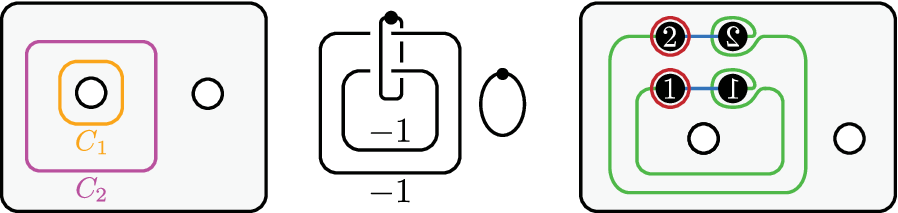}
\caption{Left: Vanishing cycles $C_1$ and $C_2$ on $\Sigma_{0,3}$, Middle: The induced handlebody diagram, Right: The induced $(2,2;0,3)$-relative trisection diagram.}
\label{fig:g=2b=3b}
\end{figure}
\begin{figure}[!htbp]
\centering
\includegraphics[scale=0.75]{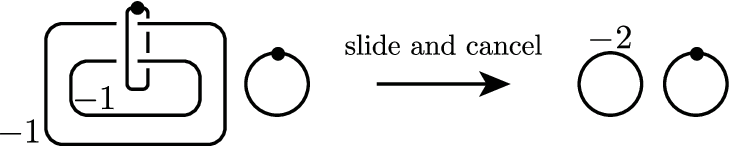}
\caption{A handle move.}
\label{fig:kd-E-2-S1xD3}
\end{figure}

Finally, assume that $b=4$. 
Up to exchanging $C_1$ and $C_2$, there are three possibilities for the pair of vanishing cycles:
\begin{enumerate}
    \item both $C_1$ and $C_2$ are boundary-parallel,
    \item $C_1$ is boundary-parallel and $C_2$ is not boundary-parallel, or
    \item neither $C_1$ nor $C_2$ is boundary-parallel.
\end{enumerate}

In case~(1), we may assume that $(\Sigma_{0,4};C_1,C_2)$ is as shown in Figures~\ref{fig:g=2b=4a1}~(left) and \ref{fig:g=2b=4a2}~(left).
The total space is diffeomorphic to either $S^1\times D^3$ or $E_{-2}\natural^2(S^1\times D^3)$.

\begin{figure}
\centering
\includegraphics[scale=0.75]{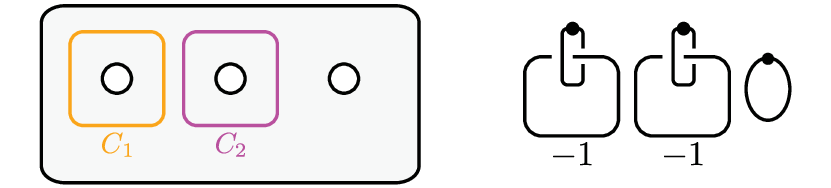}
\caption{Left: Vanishing cycles $C_1$ and $C_2$ on $\Sigma_{0,4}$, Right: The induced handlebody diagram.}
\label{fig:g=2b=4a1}
\end{figure}

\begin{figure}
\centering
\includegraphics[scale=0.75]{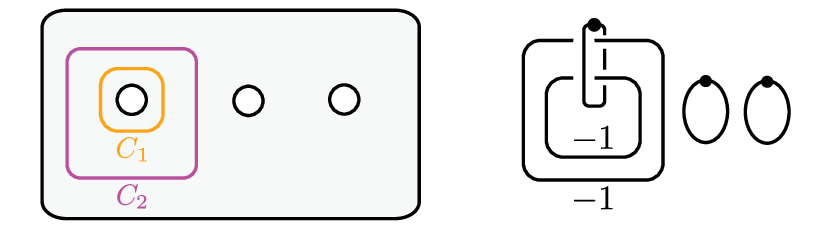}
\caption{Left: Vanishing cycles $C_1$ and $C_2$ on $\Sigma_{0,4}$, Right: The induced handlebody diagram.}
\label{fig:g=2b=4a2}
\end{figure}

In case~(2), we may assume that $(\Sigma_{0,4};C_1,C_2)$ is as shown in Figure~\ref{fig:g=2b=4b}~(left).
The total space is diffeomorphic to $D^4$.

\begin{figure}
\centering
\includegraphics[scale=0.75]{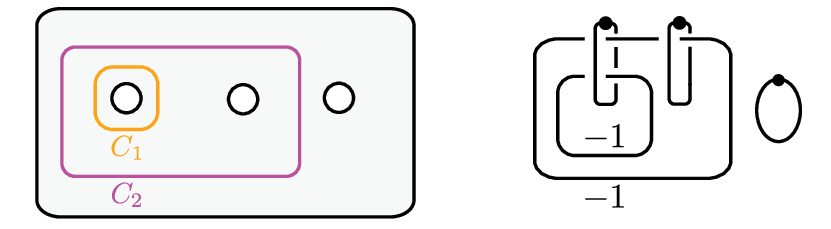}
\caption{Left: Vanishing cycles $C_1$ and $C_2$ on $\Sigma_{0,4}$, Right: The induced handlebody diagram.}
\label{fig:g=2b=4b}
\end{figure}

In case~(3), we can describe the configuration of non-boundary-parallel simple closed curves on $\Sigma_{0,4}$ by the pillowcase description, which allows us to write the vanishing cycles $C_1$ and $C_2$ as the slopes $1/0$ and $p/q$, respectively (see e.g.~\cite{CM09,MZ17}).
The corresponding PALF is shown in Figure~\ref{fig:V_pq_pillowcase}, and its total space is $V_{p,q}$.
\end{proof}

We easily get the fundamental group and the first homology group of $V_{p,q}$. 
\begin{proposition}
\label{prop:fundamental_group}
    Let $p\in\Z$ and $q\in\Z_{\ge 0}$ be coprime integers.
    Then $\pi_1(V_{p,q}) \cong G\bigl(K(p/q)\bigr)$, where $G\bigl(K(p/q)\bigr)$ is the link group of the $2$-bridge link $K(p/q)$.
    Thus, it follows that
    \[ H_1(V_{p,q}) \cong
    \begin{cases}
    \mathbb{Z} & (q : \text{ odd}),\\[4pt]
    \mathbb{Z} \oplus \mathbb{Z} & (q : \text{ even}).
    \end{cases}
    \]
    by the Hurewicz theorem applying to $\pi_1(V_{p,q})$.
\end{proposition}

\begin{proof}
For a PALF $(\Sigma_{0,4};\, C_{1}, C_2)$, we get the following by applying the van Kampen's theorem,
\begin{equation*}
\pi_{1}\bigl(X(\Sigma_{0,4};\,C_{1}, C_2) \bigr)
\cong
\pi_{1}(\Sigma_{0,4})\big/\bigl\langle\!\bigl\langle C_{1}, C_2\bigr\rangle\!\bigr\rangle,
\end{equation*}
where $\langle\!\langle C_{1}, C_2 \rangle\!\rangle$ denotes the normal closure of the conjugacy classes represented by the curves $C_{1}$ and $C_2$ in $\Sigma_{0,4}$.
Since $V_{p,q}=X(\Sigma_{0,4};\,C_{1},C_{2})$, we have
$\pi_{1}(V_{p,q}) \cong \pi_{1}(\Sigma_{0,4})\big/\bigl\langle\!\bigl\langle C_{1},C_{2}\bigr\rangle\!\bigr\rangle$.
On the other hand, it follows that $G\bigl(K(p/q)\bigr) \cong \pi_{1}(\Sigma_{0,4})\big/\bigl\langle\!\bigl\langle C_{1},C_{2}\bigr\rangle\!\bigr\rangle$ by the van Kampen's theorem.
This is stated in \cite{ORS08}.
Thus, we have $\pi_1(V_{p,q}) \cong G\bigl(K(p/q)\bigr)$.
\end{proof}

We determine the relative trisection genus of $V_{p,q}$.

\begin{proposition}\label{prop:reltrisection_genus}
If $q\neq1$, then the relative trisection genus of $V_{p,q}$ is $2$.
That is, $V_{p,q}$ admits no relative trisections of genus at most $1$.
\end{proposition}


\begin{proof}
First, we show that if $V_{p,q}$ admits a $(1,k;p,b)$-relative trisection, then $(k,p,b)=(2,0,3)$.
By the definition of a relative trisection, the parameters satisfy $k,p \geq 0$, $b \geq 1$, and
\begin{equation}
2p + b - 1 \leq k \leq p + b. \label{eq:ineq}
\end{equation}
In addition, since $\chi(V_{p,q})=0$, Proposition~\ref{prop:Euler-rel} implies that
\begin{equation}
-3k+3p+2b=0. \label{eq:chi}
\end{equation}
The inequality~\eqref{eq:ineq} implies $2p+b-1 \leq p+b$, which simplifies to $p\leq1$. Thus, we have $p \in \{0,1\}$.
If $p=1$, then \eqref{eq:ineq} gives $k=b+1$. 
Substituting these values into \eqref{eq:chi}, we obtain
\begin{equation*}
-3(b+1)+3\cdot1+2b = 0,
\end{equation*}
which yields $b=0$. This contradicts the condition $b \geq 1$.

Now suppose $p=0$. In this case, \eqref{eq:ineq} implies $k \in \{b-1, b\}$.
If $k=b$, then \eqref{eq:chi} gives
\begin{equation*}
-3b + 3\cdot0 + 2b = 0,
\end{equation*}
hence $b=0$, again a contradiction.
If $k=b-1$, then \eqref{eq:chi} gives
\begin{equation*}
-3(b-1) + 3\cdot0 + 2b = 0,
\end{equation*}
which leads to $b=3$.
Consequently, we conclude that $(k,p,b) = (2,0,3)$.

Next, we determine the $4$-manifolds that admit a $(1,2;0,3)$-relative trisection.
Since each curve system of a $(1,2;0,3)$-relative trisection diagram consists of a single component, we cannot perform handleslides within any curve system.
Thus, the diffeomorphism classes of $(1,2;0,3)$-relative trisections correspond to the equivalence classes of $(1,2;0,3)$-relative trisection diagrams, up to surface diffeomorphisms and isotopies of the individual curves.

Let $\mathcal{D}=(\Sigma_{1,3};\alpha,\beta,\gamma)$ be a $(1,2;0,3)$-relative trisection diagram.
By the definition of a relative trisection diagram (see~\cite{CGPC18}), each of the triples $(\Sigma_{1,3};\alpha,\beta)$, $(\Sigma_{1,3};\beta,\gamma)$, and $(\Sigma_{1,3};\gamma,\alpha)$ can be modified to the standard diagram shown in Figure~\ref{fig:1203std} by self-diffeomorphisms of $\Sigma_{1,3}$ and isotopies of the individual curves.
Note that any surface diffeomorphism preserves the intersection number of curves. 
That is, for any diffeomorphism $f:\Sigma\to\Sigma$ and any simple closed curves $c,c'\subset\Sigma$ intersecting transversely, we have $|c \cap c'| = |f(c) \cap f(c')|$.
Hence, by isotoping each curve if necessary, we can assume that each of the intersections $\alpha \cap \beta$, $\beta \cap \gamma$, and $\gamma \cap \alpha$ is a single point.
\begin{figure}[!htbp]
\centering
\includegraphics[scale=0.75]{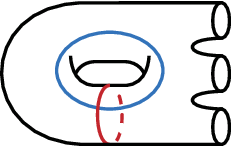}
\caption{$(1,2;0,3)$-standard diagram.}
\label{fig:1203std}
\end{figure}

Let $\widehat{\mathcal{D}}=(\Sigma_1;\alpha,\beta,\gamma)$ be the diagram obtained from $\mathcal{D}$ by capping off each of the three boundary components with a disk.
By Lemma~3.1 of \cite{T24}, $\widehat{\mathcal{D}}$ is a $(1,0)$-trisection diagram.
Up to diffeomorphisms of the underlying surface $\Sigma_1$, the capped diagram $\widehat{\mathcal{D}}$ coincides with one of the two standard $(1,0)$-trisection diagrams shown in Figures~\ref{fig:td-+cp2} and \ref{fig:td--cp2}, namely $\mathcal{D}_{+\mathbb{C} P^2}$ or $\mathcal{D}_{-\mathbb{C} P^2}$.
These are the standard $(1,0)$-trisection diagrams of $\pm \mathbb{C}P^2$.
Thus, the $(1,2;0,3)$-relative trisection diagram $\mathcal{D}$ is obtained from $\mathcal{D}_{+\mathbb{C}P^2}$ or $\mathcal{D}_{-\mathbb{C} P^2}$ by removing three disjoint open disks.
\begin{figure}[h]
\centering
\begin{minipage}[b]{0.49\columnwidth}
    \centering
    \includegraphics[scale=0.75]{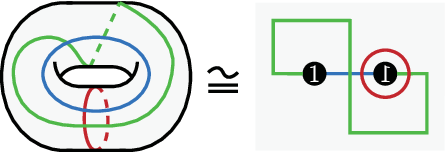}
    \caption{$\mathcal{D}_{+\mathbb{C} P^2}$}
    \label{fig:td-+cp2}
\end{minipage}
\begin{minipage}[b]{0.49\columnwidth}
    \centering
    \includegraphics[scale=0.75]{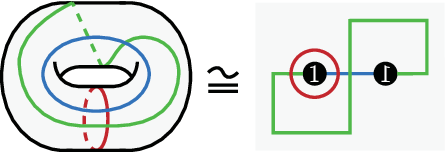}
    \caption{$\mathcal{D}_{-\mathbb{C} P^2}$}
    \label{fig:td--cp2}
\end{minipage}
\end{figure}

We now focus on the case of $\mathcal{D}_{+\mathbb{C}P^2}$.
Note that the same argument applies to the case of $\mathcal{D}_{-\mathbb{C}P^2}$.
The three simple closed curves of $\mathcal{D}_{+\mathbb{C}P^2}$ decompose the torus into three regions.
Allowing for multiple disks to be removed from the same region, we consider all possible placements of the three removed disks.
This leads to the ten cases shown in Figures~\ref{fig:td-cp2dotdot}, \ref{fig:td-dot}, and \ref{fig:td-E2dot}.
The diagrams in Figure~\ref{fig:td-cp2dotdot} represent $\mathbb{C} P^2 \# (\natural^2 (S^1 \times D^3))$.
The diagrams in Figure~\ref{fig:td-dot} represent $S^1\times D^3$.
The diagrams in Figure~\ref{fig:td-E2dot} represent the connected sum of $S^1 \times D^3$ and the $D^2$-bundle over $\mathbb{R}P^2$ with Euler number $1$.
\begin{figure}[!htbp]
\centering
\includegraphics[scale=0.75]{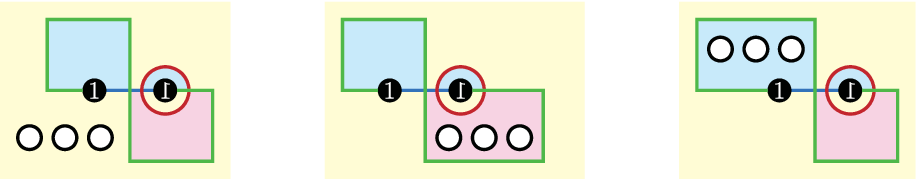}
\caption{$(1,2;0,3)$-relative trisection diagrams of $\mathbb{C} P^2 \# (\natural^2 (S^1 \times D^3))$.}
\label{fig:td-cp2dotdot}
\end{figure}
\begin{figure}[!htbp]
\centering
\includegraphics[scale=0.75]{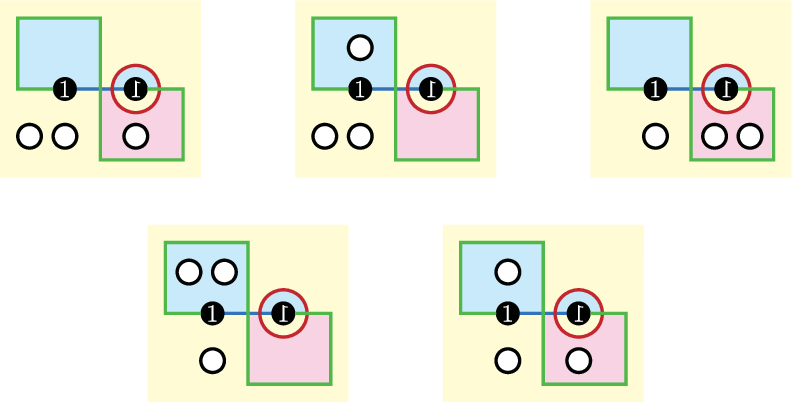}
\caption{$(1,2;0,3)$-relative trisection diagrams of $S^1\times D^3$.}
\label{fig:td-dot}
\end{figure}
\begin{figure}[!htbp]
\centering
\includegraphics[scale=0.75]{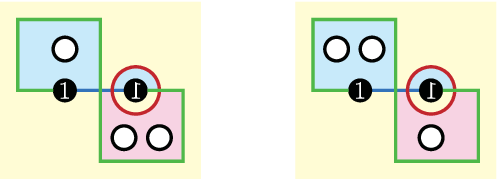}
\caption{$(1,2;0,3)$-relative trisection diagrams of the connected sum of $S^1 \times D^3$ and the $D^2$-bundle over $\mathbb{R}P^2$ with Euler number $1$.}
\label{fig:td-E2dot}
\end{figure}

Finally, we prove that if $q \neq 1$, then $V_{p,q}$ is not diffeomorphic to any of the $4$-manifolds listed above or to any of their orientation reversals.
By Proposition~\ref{prop:fundamental_group}, $H_1(V_{p,q}) \cong \mathbb{Z}$ if $q$ is odd, and $H_1(V_{p,q}) \cong \mathbb{Z} \oplus \mathbb{Z}$ if $q$ is even.
By comparing these homology groups with those of the manifolds listed above, it follows that $V_{p,q}$ is diffeomorphic neither to the connected sum of $S^1 \times D^3$ and the $D^2$-bundle over $\mathbb{R}P^2$ with Euler number $1$ nor its orientation reversal.

Suppose that $V_{p,q} \cong \pm{\mathbb{C}P^2}\#(\natural^2 (S^1 \times D^3))$.
Then we have $G({K(p/q)}) \cong \pi_1(V_{p,q}) \cong F_2$, where $F_2$ is the free group of rank $2$.
Hence, $K(p/q)$ is the 2-component unlink, which implies $(p,q) = (1,0)$.
However, this contradicts the fact that $V_{1,0} \cong E_{-2} \natural (\natural^2 (S^1 \times D^3))$.

Suppose that $V_{p,q} \cong S^1 \times D^3$.
Then we have $G({K(p/q)}) \cong \pi_1(V_{p,q}) \cong \mathbb{Z}$.
It follows that $K(p/q)$ is the unknot, which implies $q=1$. 
This contradicts the assumption $q \neq 1$.
\end{proof}

\section{Classification of $V_{p,q}$}

A natural question is to determine for which pairs $(p,q)$ and $(p',q')$ the Stein surfaces $V_{p,q}$ and $V_{p',q'}$ are diffeomorphic.
In this section, we give a complete classification of $V_{p,q}$ in the case where $q$ is odd.
Let $\Sigma$ be an oriented surface with boundary, and let $C_1, C_2\subset\Sigma$ be simple closed curves.

Let $X(\pm \Sigma; C_1, C_2)$ denote the $4$-manifold obtained from $\Sigma\times D^2$ by attaching $2$-handles along $C_1$ and $C_2$ with framing $\mp 1$ relative to the surface framing.
Observe that $-X(\pm \Sigma; C_1, C_2)$ is diffeomorphic to $X(\mp \Sigma; C_1, C_2)$.
We write the PALF presentations of $V_{p,q}$ and $V_{p',q'}$ as $X(\Sigma_{0,4};C_1,C_2)$ and $X(\Sigma_{0,4};C'_1,C'_2)$, respectively. Here $C_1$ and $C'_1$ are the vanishing cycles corresponding to the slope $1/0$, while $C_2$ and $C'_2$ are the vanishing cycles corresponding to the slopes $p/q$ and $p'/q'$, respectively.

\begin{lemma}\label{lem:conjugate}
Let $\Sigma$ be an oriented, compact, and connected surface with nonempty boundary, and let $C_1, C_2 \subset\Sigma$ be essential simple closed curves.
For any diffeomorphism $\phi:\Sigma\to\Sigma$, possibly orientation-reversing, we have 
$X(\Sigma;C_1,C_2) \cong X(\Sigma;\phi(C_1), \phi(C_2))$.
\end{lemma}

\begin{remark}
    If $\phi$ is an orientation-preserving diffeomorphism, then $(\Sigma;\phi(C_1),\phi(C_2))$ is a global conjugate of the PALF $(\Sigma;C_1,C_2)$.
    Lemma~\ref{lem:conjugate} also covers the case in which $\phi$ is orientation-reversing. In this case, the PALF $(\Sigma; \phi(C_1), \phi(C_2))$ is not necessarily a global conjugate of $(\Sigma; C_1, C_2)$, but its total space is still diffeomorphic to that of $(\Sigma; C_1, C_2)$.
\end{remark}

\begin{proof}[Proof of Lemma~\ref{lem:conjugate}]
Consider $\Phi:=\phi\times\mathrm{id}_{D^2}:\Sigma\times D^2\to\Sigma\times D^2$.
Its restriction to $\Sigma\times S^1 \subset \partial(\Sigma\times D^2)$ is a self-diffeomorphism and sends each attaching circle $C_i\times\{\theta_i\}$ to $\phi(C_i)\times\{\theta_i\}$.
Suppose first that $\phi$ preserves the orientation.
Recall that the total space of the PALF $(\Sigma;\phi(C_1),\phi(C_2))$ is obtained from $\Sigma\times D^2$ by attaching $2$-handles along $\phi(C_1), \phi(C_2)$ (see Section~\ref{subsec:Lefschetz fibrations}).
If $\phi$ preserves the orientation, then the $2$-handles are attached with the surface framing minus one.
Therefore $\Phi$ extends over the attached $2$-handles, yielding the desired orientation-preserving diffeomorphism between the total spaces.
If $\phi$ is an orientation-reversing diffeomorphism of $\Sigma$, then $X(\Sigma;C_1, C_2) \cong -X(-\Sigma;\phi(C_1), \phi(C_2)) \cong X(\Sigma;\phi(C_1), \phi(C_2))$.
In either case, we obtain $X(\Sigma;C_1,C_2) \cong X(\Sigma;\phi(C_1), \phi(C_2))$ as claimed.
\end{proof}
 
In the next proposition, we prove that $V_{p,q}$ is diffeomorphic to $V_{p',q'}$ when $q=q'$ and either $p \equiv \pm p' \pmod{q}$ or $pp' \equiv \pm 1 \pmod{q}$.
We recall the corresponding between slopes on $\Sigma_{0,4}$ and primitive vectors in $H_1(T^2;\mathbb{Z})$ for the proof.
Let $\iota$ be the involution on $T^2$ defined by $\iota(x)=-x$, and let
$\pi:T^2=\mathbb{R}^2/\mathbb{Z}^2 \longrightarrow T^2/\langle \iota\rangle \cong S^2$ be the quotient map.
The quotient map $\pi$ is a double branched cover over the four fixed points of the involution $\iota$, and we regard $\Sigma_{0,4}$ as the complement of these four branch points in $T^2/\langle \iota\rangle$. 
Any self-diffeomorphism of $\Sigma_{0,4}$ lifts to a homeomorphism of $T^2$, unique up to composition with the involution $\iota$.
Hence every self-diffeomorphism of $\Sigma_{0,4}$ induces an automorphism of $H_1(T^2;\mathbb Z)\cong\mathbb Z^2$.
Conversely, every diffeomorphism of $T^2$ that commutes with $\iota$ descends to a diffeomorphism of $\Sigma_{0,4}$.

An essential simple closed curve $l \subset \Sigma_{0,4}$ lifts under $\pi$ to two disjoint parallel essential curves on $T^2$, which are exchanged by $\iota$.
Each component determines a primitive vector $v\in H_1(T^2;\mathbb Z)$, since every essential simple closed curve on $T^2$ represents a primitive homology class.
Thus $l$ corresponds to a primitive vector of $H_1(T^2;\mathbb Z)$, well-defined up to overall sign.
After fixing the standard basis of $H_1(T^2;\mathbb{Z})$, we therefore identify the slope $p/q$ with the primitive vector
$
\begin{bmatrix}
p\\
q
\end{bmatrix}$,
where $p$ and $q$ are coprime integers are not both zero. In particular, the slopes $1/0$ and $0/1$ correspond to
$
\begin{bmatrix}
1\\
0
\end{bmatrix}
\ \text{and}\ 
\begin{bmatrix}
0\\
1
\end{bmatrix},
$
respectively. 
Under this identification, the action of $A\in \GL_2(\mathbb{Z})$ on slopes is given by $v \longmapsto Av$.
    
\mainVpq*

\begin{proof}
Assume that either $p\equiv \pm p'\pmod q$ or $pp'\equiv \pm 1\pmod q$.

\begin{itemize}
\item[(1)] Suppose first that $p\equiv \epsilon p' \pmod{q}$ for some $\epsilon\in\{1,-1\}$.
Then there exists $m\in\Z$ such that $p+mq=\epsilon p'$.
Set
\[
A=\begin{bmatrix}1&m\\0&\epsilon\end{bmatrix}\in GL_2(\Z).
\]
Then
\[
A\colvec{1\\0}=\colvec{1\\0} \quad \text{and} \quad
A\colvec{p\\q}=\epsilon \colvec{p'\\ q'}
\]
since $q=q'$.
The second equality means that $A$ sends the slope $p/q$ to $p'/q'$,
{while} the first fixes the slope $1/0$.

\item[(2)] Next, suppose that $pp'\equiv \epsilon \pmod{q}$ for some $\epsilon\in\{1,-1\}$.
Then there exists $m\in\Z$ such that $pp'+mq=\epsilon$.
Set
\[
A=\begin{bmatrix}p'&m\\ q&-p\end{bmatrix}\in GL_2(\Z).
\]
Then
\[
A\colvec{1\\0}=\colvec{p'\\q'} \quad \text{and} \quad
A\colvec{p\\q}=\epsilon \colvec{1\\0}
\]
since $q=q'$.
Thus $A$ sends the slope $1/0$ to $p'/q'$ and sends $p/q$ to $1/0$ on $\Sigma_{0,4}$.
\end{itemize}

In either case, there exists a self-homeomorphism
$\phi:\Sigma_{0,4}\to\Sigma_{0,4}$ such that either
\[
\phi(C_1)= C'_1\ \text{ and }\ \phi(C_2)= C'_2,
\ \text{or}\ \phi(C_1)= C'_2\ \text{ and }\ \phi(C_2)= C'_1.
\]
By Lemma~\ref{lem:conjugate},
$X(\Sigma_{0,4}; C_1,C_2) \cong X(\Sigma_{0,4}; \phi(C_1),\phi(C_2))$.
If $\phi$ exchanges the two vanishing cycles, as in case (2), then we may apply a Hurwitz move and a global conjugation to exchange their order.
This does not change the total space.
Therefore, we obtain
$V_{p,q} \cong V_{p',q}$.
\end{proof}

Even in the case when $\phi$ is orientation-reversing, Lemma~\ref{lem:conjugate} shows that the total spaces are still diffeomorphic.
We now prove that, when $q$ is odd, these congruence conditions are not only sufficient but also necessary for $V_{p,q}$ and $V_{p',q'}$ to be diffeomorphic.

\mainclass*

\begin{proof}
If $V_{p,q}$ is diffeomorphic to $V_{p',q'}$, then $\pi_1(V_{p,q})$ is isomorphic to $\pi_1(V_{p',q'})$.
By Proposition~\ref{prop:fundamental_group}, we obtain $G\bigl(K(p/q)\bigr)\cong G\bigl(K(p'/q')\bigr)$.
Since $q$ and $q'$ are odd, both $K(p/q)$ and $K(p'/q')$ are $2$-bridge knots and hence prime.
Therefore, the exteriors of $K(p/q)$ and $K(p'/q')$ are homeomorphic \cite{GL89}. 
By Schubert's classification of $2$-bridge knots, this is equivalent to the desired conditions. 
The converse follows immediately from Proposition~\ref{prop:homeo_on_fiber}.
\end{proof}

\section*{Acknowledgements}
The authors are grateful to thank Masaharu Ishikawa for his helpful comments. 
The second author was partially supported by JSPS KAKENHI Grant Number 24KJ1561.

\end{document}